\scrollmode

\documentclass[11pt]{amsart}
\usepackage{amssymb}
\usepackage[all]{xy}
\usepackage{graphicx}

\theoremstyle{plain}
\newtheorem{thm}{Theorem}
\newtheorem{coro}[thm]{Corollary}
\newtheorem*{thmA}{Theorem A}
\newtheorem*{corB}{Corollary B}
\newtheorem{lem}[thm]{Lemma}

\newtheorem{prop}[thm]{Proposition}

\theoremstyle{remark}

\newtheorem{rem}[thm]{Remark}

\numberwithin{thm}{section}
\numberwithin{equation}{section}

\theoremstyle{definition}
\newtheorem*{defi}{Definition}

\newcommand{\F}{\mathbb{F}}

\newcommand{\N}{\mathbb{N}}
\newcommand{\Q}{\mathbb{Q}}

\newcommand{\rk}{\mathrm{rk}}
\newcommand{\dd}{\mathrm{d}}

\DeclareMathOperator{\Gal}{Gal}

\begin{document}

\title[Finite quotients of Galois pro-$p$ groups]{Finite quotients of Galois pro-$p$ groups \\ and rigid fields}
\author{C. Quadrelli}
\date{\today}
\address{Department of Mathematics, Arts and Crafts\\ University of Milano-Bicocca\\ Ed.~U5, Via R.Cozzi 53\\
20125 Milano, Italy}
\email{c.quadrelli1@campus.unimib.it}

\begin{abstract}
For a prime number $p$, 
we show that if two certain canonical finite quotients of a finitely generated Bloch-Kato pro-$p$ group $G$ coincide,
then $G$ has a very simple structure, i.e., $G$ is a $p$-adic analytic pro-$p$ group (see Theorem~A).
This result has a remarkable Galois-theoretic consequence: if the two corresponding canonical finite extensions
$F^{(3)}/F$ and $F^{\{3\}}/F$ of a field $F$ --
with $F$ containing a primitive $p$-th root of unity -- coincide, then $F$ is $p$-rigid (see Corollary~B).
The proof relies only on group-theoretic tools, and on certain properties of Bloch-Kato pro-$p$ groups.
This paper will appear on the {\it Annales math\'ematiques du Qu\'ebec}.
\end{abstract}


\maketitle

\section{Introduction}\label{s:intro}

Let $p$ be a prime number, and let $G$ be a pro-$p$ group.
The Frattini subgroup $\Phi(G)$ of $G$ is the closed subgroup of $G$ generated by the $p$-powers
and the commutators of the elements of $G$.
In particular, the quotient $G/\Phi(G)$ is an elementary abelian $p$-group.
Let $\Phi_2(G)$ be the Frattini subgroup of the Frattini subgroup of $G$, i.e., $\Phi_2(G)=\Phi(\Phi(G))$.

Also, let $P_n(G)$, $n\geq1$, denote the $p$-descending central series of $G$.
In particular, one has $P_2(G)=\Phi(G)$ and $P_3(G)=\Phi(G)^p[G,\Phi(G)]\supseteq \Phi_2(G)$.
For the class of finitely generated Bloch-Kato pro-$p$ groups, we prove the following result.

\begin{thmA}
 One has the equality $\Phi_2(G)=P_3(G)$ if, and only if, $G$ is $p$-adic analytic.
\end{thmA}

In this case the group $G$ has a very simple structure, as it is meta-abelian
and it is possible to provide an explicit presentation for $G$ (cf. \cite[Theorem~4.6]{blochkato}).

One has also the following Galois-theoretic consequence.
Let $F$ be a field containing a primitive $p$-th root of unity.
By $F^\times$ we denote the (multiplicative) group of non-zero elements of $F$.
We consider the Galois extension $F^{(3)}$ of $F$ obtained by first taking $F^{(2)}$ to be the compositum
over $F$ of all extensions of $F$ of degree $p$, and then taking $F^{(3)}$ to be the compositum over $F^{(2)}$
of all the extensions of $F^{(2)}$ of degree $p$ that are Galois over $F$.
We also denote by $F^{\{3\}}$ the compositum over $F^{(2)}$ of all extensions of $F^{(2)}$ of degree $p$
(cf. \cite[\S~2.3]{cmq:fast}). Thus
\[F^{\{3\}}={(F^{(2)})}^{(2)}.\]
Then one may characterize those fields $F$ with the property that $F^{(3)}=F^{\{3\}}$.
In fact, from Theorem~A we shall obtain the following result.

\begin{corB}
Let $F$ be a field containing a primitive $p$-th root of unity, and assume that the quotient $F^\times/(F^\times)^p$
is finite. (Assume further that $\sqrt{-1}\in F$ if $p=2$).
Then $F^{(3)}=F^{\{3\}}$ if, and only if, $F$ is $p$-rigid;
\end{corB}

(For the definition of $p$-rigid field, see Section~4.)

Bloch-Kato pro-$p$ groups were introduced in \cite{bk1} and studied first in \cite{blochkato}.
A Bloch-Kato pro-$p$ group is a pro-$p$ group which satisfies the conclusion of the Rost-Voevodsky theorem
(formerly known as the Bloch-Kato conjecture), i.e., such that the cohomology ring of every closed subgroup of $G$
with coefficients in the finite field $\F_p$ is a quadratic algebra over $\F_p$.
For example, absolute Galois groups of fields which are pro-$p$ and Galois groups of the maximal $p$-extension
of certain fields are Bloch-Kato pro-$p$ groups.
Thus, a Bloch-Kato pro-$p$ group is a very natural ``candidate'' for being realized as absolute Galois group,
and this shows the relevance of Bloch-Kato pro-$p$ groups for Galois theory.

The problem to characterize a field $F$ yielding the equality 
\begin{equation}
 \label{eq:equality fields}
F^{(3)}=F^{\{3\}}
\end{equation}
arises rather naturally, and the case when equality (\ref{eq:equality fields}) holds is considered very significant
in field theory.
Indeed, such problem has been widely studied in the past: in the case $p=2$ Corollary~B was proved
in \cite[Theorem~3.1]{agkm}, with arguments which make use of Galois cohomology,
and later in \cite[Theorem~A]{leepsmith}, with arguments relying on the theory of quadratic forms.
For $p$ odd, Corollary~B was proved in \cite[Theorem~A]{cmq:fast}, and the proof relies on certain properties
of Bloch-Kato pro-$p$ groups, together with an essential arithmetic argument (cf. \cite[Theorem~4.3]{cmq:fast}).

The above results provide a motivation for the paper, as Theorem~A is the ``group-theoretic translation'',
and it is in fact more genaral, as it holds for Bloch-Kato pro-$p$ groups,
and not only for Galois groups of maximal $p$-extensions.
Moreover, part of the interest of this result lies in the fact that the proof is purely group-theoretical,
and it does not rely on results form field theory.
Further, the proof makes use of the Zassenhaus filtration of pro-$p$ groups,
which is gaining increasing importance as tool for the study of Galois groups
(see, e.g., \cite{efrat:zassenhaus} and \cite{jantan}).

The paper is organized as follows.
In the second section, we state a number of properties on pro-$p$ groups and on their descending series.
In section~3 we prove Theorem~A, and in section~4 we provide the ``arithmetic translation''
of Theorem~A, and we prove Corollary~B.

This paper will be published on the {\it Annales math\'ematiques du Qu\'ebec}.


\section{Preliminaries on pro-$p$ groups}

Throughout this paper, subgroups of pro-$p$ groups are assumed to be closed (in the pro-$p$ topology),
and every generator is to be intended as topological generator.
In particular, given two (closed) subgroups $H_1$ and $H_2$ of a pro-$p$ group $G$, the 
subgroup $[H_1,H_2]$ is the (closed) subgroup of $G$ generated by the commutators $[g_1,g_2]$, with
$g_i\in H_i$ for $i=1,2$.
Also, for a positive integer $n$, $G^n$ denotes the (closed) subgroup of $G$ generated by the $n$-powers
of the elements of $G$.

For a finitely generated pro-$p$ group $G$, let $\dd(G)$ denote the minimal number of generators of $G$.
In particular, $\dd(G)$ is the dimension of the quotient $G/\Phi(G)$ as vector space over the finite field $\F_p$
(cf. \cite[Prop.~1.14]{ddsms}).
Then, one defines the rank of a pro-$p$ group $G$ to be the number
\[\rk(G)=\sup\{\dd(H) \:|\: H\leq G\text{ closed}\}\in\N\cup\{\infty\}\]
(cf. \cite[Definition~3.12]{ddsms}).

For a pro-$p$ group $G$, the lower $p$-central series of $G$ is the series $P_n=P_n(G)$, $n\geq1$,
of characteristic subgroups defined by $P_1=G$ and
\[ P_{n+1}=P_n^p[G,P_n]. \]
In particular, one has that $P_2(G)$ is the Frattini subgroup $\Phi(G)$, and $[P_i,P_j]\leq P_{i+j}$
for every $i,j\geq1$.
Moreover, if $G$ is finitely generated, then
the lower $p$-central series is a base of neighbourhoods of 1 in $G$ (cf. \cite[Prop.~1.16]{ddsms}).

\begin{defi}
 A pro-$p$ group $G$ is said to be {\it powerful} if $G/G^p$ is abelian, if $p$ is odd, or if $G/G^4$ is abelian,
if $p=2$.
\end{defi}

In particular, one has the following (cf. \cite[Theorems~3.6, 3.8]{ddsms}).

\begin{prop}
 \label{prop:powerful}
Let $G$ be a powerful pro-$p$ group.
\begin{itemize}
 \item[(1)] $P_n(G)=G^{p^{n-1}}$ for every $n\geq1$.
 \item[(2)] if $G$ is finitely generated, then $\rk(G)=\dd(G)$.
\end{itemize}
\end{prop}

Another important descending series of pro-$p$ groups is the {\it Zassenhaus filtration}.
For an arbitrary group $G$, the Zassenhaus filtration of $G$ is the series $D_n=D_n(G)$, $n\geq1$,
of characteristic subgroups defined by $D_1=G$ and
\begin{equation}
 \label{eq:zassenhaus definition}
 D_n=D_{\lceil n/p\rceil}^p\prod_{i+j=n}\left[D_i,D_j\right],
\end{equation}
where $\lceil n/p\rceil$ is the least integer $m$ such that $mp\geq n$.
In particular, the Zassenhaus filtration is the fastest descending series starting at $G$
such that $[D_i,D_j]\leq D_{i+j}$ and $D_i^p\leq D_{ip}$ for every $i,j\geq1$.
For computational purposes, one has the formula
\begin{equation}
 \label{eq:Lazard formula}
 D_n=\prod_{ip^h\geq n}\gamma_i(G)^{p^h},
\end{equation}
established by M.~Lazard (cf. \cite[Theorem~11.2]{ddsms}),
where the $\gamma_i(G)$'s are the elements of the descending central series of $G$
(i.e., $\gamma_1(G)=G$ and $\gamma_{i+1}(G)=[G,\gamma_i(G)]$ for every $i\geq1$).
Thus, if $G$ is a (pro-)$p$ group, then $D_2(G)$ is the Frattini subgroup $\Phi(G)$.

For the Zassenhaus filtration of a pro-$p$ group,
one has the following remarkable result (cf. \cite[Theorem~11.4]{ddsms}).

\begin{thm}
 \label{thm:zassenhaus and rank}
Let $G$ be a finitely generated pro-$p$ group.
Then $G$ has finite rank if, and only if, $D_n(G)=D_{n+1}(G)$ for some $n\geq1$.
\end{thm}

\begin{defi}
A topological group $G$ is a {\it $p$-adic analytic group} if $G$ has the structure of analytic manifold 
over the field of $p$-adic numbers $\Q_p$ with the properties
\begin{itemize}
 \item[(1)] the multiplication function $G\times G\to G$ given by $(x,y)\mapsto xy$ is analytic;
 \item[(2)] the inversion function $G\to G$ defined by $x\mapsto x^{-1}$ is analytic.
\end{itemize}
\end{defi}

Powerful pro-$p$ groups and $p$-adic analytic groups are tightly related.
Indeed, a topological group $G$ has the structure of a $p$-adic analytic group if, and only if,
$G$ has an open subgroup which is a powerful finitely generated pro-$p$ group (cf. \cite[Theorem~8.1]{ddsms}).
In the case of Bloch-Kato pro-$p$ groups, $p$-adic analytic groups have a rather simple structure,
as stated by the following (cf. \cite[Theorem~4.8]{blochkato}).

\begin{thm}
 \label{thm:bk}
Let $G$ be a finitely generated Bloch-Kato pro-$p$ group, and assume furhter that $G$ is torsion-free, if $p=2$.
The following are equivalent.
\begin{itemize}
 \item[(1)] $G$ has finite rank.
 \item[(2)] $G$ is $p$-adic analytic.
 \item[(3)] $G$ is powerful.
 \item[(4)] $G$ has a presentation
\begin{equation}
 \label{eq:presentation}
G=\left\langle\sigma,\tau_1,\ldots,\tau_d\:\left|\:\sigma\tau_i\sigma^{-1}=\tau_i^{1+p^k},
\tau_i\tau_j=\tau_j\tau_i \ \forall\:i,j\right.\right\rangle,
\end{equation}
with $d=\dd(G)-1$, for some $k\geq1$ ($k\geq2$, if $p=2$).
\end{itemize}
\end{thm}


\section{Proof of Theorem A}

\begin{lem}
 \label{lem:proof}
If $G$ is a powerful Bloch-Kato group, then $\Phi_2(G)=P_3(G)$.
\end{lem}

\begin{proof}
Recall first that if $G$ is a Bloch-Kato pro-$p$ group, then every closed subgroup of $G$ is again
a Bloch-Kato pro-$p$ group.
By Proposition~\ref{prop:powerful}, one has $\Phi(G)=G^p$ and $P_3(G)=G^{p^2}$.
Since $\rk(G)$ is finite, also $\rk(\Phi(G))$ is finite, thus $\Phi(G)$ is powerful by Theorem~\ref{thm:bk}.
Therefore,  
 \[\Phi_2(G)=\Phi(\Phi(G))=\Phi(G)^p=G^{p^2},\]
and this yields the claim.
\end{proof}

\begin{proof}[Proof of Theorem~A]
Assume that $G$ is a finitely generated $p$-adic analytic Bloch-Kato group.
Then, the claim holds by Theorem~\ref{thm:bk} and Lemma~\ref{lem:proof}.

Conversely, assume that $\Phi_2(G)=P_3(G)$.
Since $[D_2,D_2]\leq D_4$ and $D_2^p\leq D_{2p}$, one has $\Phi_2(G)=D_2^p[D_2,D_2]\leq D_4$, as $\Phi(G)=D_2$.
Moreover, one has the inclusion $\gamma_3(G)\leq P_3(G)$.
Therefore, one has the chain of inclusions
\begin{equation}
 \label{eq:inclusions}
\gamma_3(G)\leq P_3(G)=\Phi_2(G)\leq D_4.
\end{equation}

We shall split the proof of this implication in three cases.
\begin{itemize}

 \item[(1)] Assume $p>3$.
By \eqref{eq:Lazard formula}, one has
\begin{eqnarray*}
&& D_3=\prod_{ip^h\geq3}\gamma_i(G)^{p^h}=\gamma_3(G)\cdot G^p \\
&\text{and}& D_4=\prod_{ip^h\geq4}\gamma_i(G)^{p^h}=\gamma_4(G)\cdot G^p.
\end{eqnarray*}
Therefore, \eqref{eq:inclusions} implies
\[ D_3(G)=\gamma_3(G)\cdot G^p\leq P_3(G)=\Phi_2(G)\cdot G^p\leq D_4,\]
as $G^p\leq D_4$.
Thus, one has the equality $D_3=D_4$.
Hence, Theorem~\ref{thm:zassenhaus and rank} implies that $\rk(G)$ is finite, and thus by Theorem~\ref{thm:bk} 
$G$ is a $p$-adic analytic Bloch-Kato pro-$p$ group.

 \item[(2)] Assume $p=2$.
From \eqref{eq:Lazard formula} one obtains
\begin{eqnarray*}
&& D_3=\prod_{i2^h\geq3}\gamma_i(G)^{2^h}=\gamma_3(G)\cdot \gamma_2(G)^2\cdot G^4 \\
&\text{and}& D_4=\prod_{i2^h\geq4}\gamma_i(G)^{2^h}=\gamma_4(G)\cdot \gamma_2(G)^2\cdot G^4.
\end{eqnarray*}
Therefore, \eqref{eq:inclusions} implies
\[D_3=\gamma_3(G)\cdot\gamma_2(G)^2\cdot G^4\leq \Phi_2(G)\cdot\gamma_2(G)^2\cdot G^4\leq D_4,\]
as $\gamma_2(G)^2 G^4\leq D_4$.
Thus, one has the equality $D_3=D_4$.
Hence, Theorem~\ref{thm:zassenhaus and rank} implies that $\rk(G)$ is finite, and thus by Theorem~\ref{thm:bk} 
$G$ is a $p$-adic analytic Bloch-Kato pro-$p$ group.

\item[(3)] Assume $p=3$.
By \eqref{eq:Lazard formula}, one has
\begin{eqnarray*}
&& D_4=\prod_{i3^h\geq4}\gamma_i(G)^{3^h}=\gamma_4(G)\cdot\gamma_2(G)^3\cdot G^9 \\
&\text{and}& D_5=\prod_{i3^h\geq5}\gamma_i(G)^{3^h}=\gamma_5(G)\cdot\gamma_2(G)^3\cdot G^9.
\end{eqnarray*}
Therefore, from \eqref{eq:inclusions} one obtains the chain of inclusions
\[ \gamma_4(G) = [G,\gamma_3(G)] \leq [G,D_4] = [D_1,D_4] \leq D_5,\]
which implies
\[ D_4 =\gamma_4(G)\cdot\gamma_2(G)^3\cdot G^9 \leq D_5,\]
as $G^9,\gamma_2(G)^3\leq D_5$.
Thus, one has the equality $D_4=D_5$.
Hence, Theorem~\ref{thm:zassenhaus and rank} implies that $\rk(G)$ is finite, and thus by Theorem~\ref{thm:bk} 
$G$ is a $p$-adic analytic Bloch-Kato pro-$p$ group.
\end{itemize}
This establishes the theorem.
\end{proof}

Note that if $G$ is a finitely generated pro-$p$ group, then $\Phi_2(G)$ is an open subgroup of $G$.
Thus, the quotient $G/\Phi_2(G)$ is finite, and one may reduce the equality $\Phi_2(G)=P_3(G)$
to a condition on finite $p$-groups, as done in \cite[Corollary~4.15]{cmq:fast}.

\begin{coro}
 \label{coro:finite quotients}
A finitely generated Bloch-Kato pro-$p$ group $G$ is $p$-adic analytic if, and only if,
$\Phi(G)/\Phi_2(G)$ is contained in the centre of $G/\Phi_2(G)$.
\end{coro}

\begin{proof}
Assume that $G$ is $p$-adic analytic.
Then Theorem~A yields the equality $\Phi_2(G)=P_3(G)$.
Since $[G,P_2]=[P_1,P_2]\leq P_3$, one has $[G,\Phi(G)]\leq\Phi_2(G)$,
and $\Phi(G)/\Phi_2(G)$ is central in $G/\Phi_2(G)$.

Conversely, assume that $\Phi(G)/\Phi_2(G)$ is central in $G/\Phi_2(G)$.
Hence the commutator subgroup $[G,\Phi(G)]$ is contained in $\Phi_2(G)$.
Since
\[ \Phi(G)^p\leq\Phi_2(G) \quad\text{and}\quad P_3=\Phi(G)^p[G,\Phi(G)], \]
it follows that $\Phi_2(G)$ contains $P_3(G)$, and thus the two subgroups are equal.
Therefore $G$ is $p$-adic analytic by Theorem~A.
\end{proof}


\section{Proof of Corollary B}

Throughout this section, a field $F$ is always assumed to contain a primitive $p$-th root of unity
(and also $\sqrt{-1}$, if $p=2$).
Also, $F^\times$ denotes the multiplicative group of non-zero elements of $F$, and $(F^\times)^p$ is the 
subgroup of $p$-powers of $F^\times$.

\begin{defi}
Let $N$ denote the norm map $N\colon F(\sqrt[p]{a})\to F$ of the $p$-cyclic extension $F(\sqrt[p]{a})/F$.
An $p$-power-free unit $a\in F^\times$ is said to be {\it $p$-rigid} if
\[ b\in N\left(F(\sqrt[p]{a})\right) \quad\text{if, and only if,}\quad b\in \bigcup_{k=0}^{p-1}a^k(F^\times)^p \] 
for every $b\in F^\times\smallsetminus (F^\times)^p$.
The field $F$ is called {\it $p$-rigid} if every element of $F^\times\smallsetminus (F^\times)^p$ is $p$-rigid.
\end{defi}

Recall from the Introduction that $F^{(2)}=F(\sqrt[p]{F})$ is the compositum over $F$
of all extensions $F(\sqrt[p]{a})$ with $a\in F^\times$.
Also,
\begin{itemize}
 \item $F^{\{3\}}=F^{(2)}(\sqrt[p]{F^{(2)}})$ is the compositum over $F^{(2)}$ of all the extensions 
$F^{(2)}(\sqrt[p]{a})$ with $a\in (F^{(2)})^\times$;
 \item $F^{(3)}$ is the compositum over $F^{(2)}$
of all the extensions $F^{(2)}(\sqrt[p]{a})$ such that $F^{(2)}(\sqrt[p]{a})/F$ is Galois.
\end{itemize}
Therefore, both $F^{\{3\}}/F$ and $F^{(3)}/F$ are Galois extensions, and $F^{(3)}\subseteq F^{\{3\}}$
(cf. \cite[\S~2.3]{cmq:fast}).

Let $G$ be the {\it maximal pro-$p$ Galois group} of $F$, i.e., \[G=G_F(p)=\Gal(F(p)/F),\]
where $F(p)$ is the maximal $p$-extension of $F$.
Recall that the maximal pro-$p$ Galois group of a field containing a primitive $p$-th root of unity is a Bloch-Kato
pro-$p$ group (cf. \cite[\S~2]{blochkato}).

By Kummer theory, one has that the Galois group of $F^{(2)}/F$ is the quotient $G/\Phi(G)$.
Note that $G$ is finitely generated if, and only if, the quotient $F^\times/(F^\times)^p$ is finite
(and in this case $\dd(G)=\dim(F^\times/(F^\times)^p)$), as $G/\Phi(G)$ and $F^\times/(F^\times)^p$ are 
isomorphic as discrete groups of exponent $p$.
Moreover,
\begin{equation}
 \label{eq:galois groups}
\Gal(F^{(3)}/F)=G/P_3(G) \quad\text{and}\quad \Gal(F^{\{3\}}/F)=G/\Phi_2(G) 
\end{equation}
(cf. \cite[\S~4.1]{cmq:fast}, see also \cite[\S~2]{agkm}).

\begin{rem}
 In the case $p=2$, the Galois groups $\Gal(F^{(3)}/F)$ and $\Gal(F^{\{3\}}/F)$
are called $W$-group, resp. $V$-group, of the field $F$, for the relations with the Witt ring of $F$
(cf. \cite{minacspira} and \cite{agkm}).
\end{rem}

\begin{proof}[Proof of Corollary~B]
Let $G$ be the maximal pro-$p$ Galois group $G_F(p)$.
By hypothesis, $G$ is finitely generated.
Moreover, $G$ is torsion free, since we are assuming that $\sqrt{-1}\in F$ for $p=2$.

Assume first that the equality $F^{(3)}=F^{\{3\}}$ holds.
Then, by \eqref{eq:galois groups} one has also the equality $\Phi_2(G)=P_3(G)$,
and thus Theorem~A implies that $G$ is a $p$-adic analytic Bloch-Kato pro-$p$ group,
and Theorem~\ref{thm:bk} implies that $G$ is powerful.
Therefore, by \cite[Proposition~3.8]{cmq:fast} the field $F$ is $p$-rigid.

Conversely, assume that $F$ is $p$-rigid.
Then, again by \cite[Proposition~3.8]{cmq:fast} the Galoi group $G$ is powerful, and thus $p$-adic analytic
by Theorem~\ref{thm:bk}.
Therefore, Theorem~A implies the equality $\Phi_2(G)=P_3(G)$, and the equality $F^{(3)}=F^{\{3\}}$ 
follows by \eqref{eq:galois groups}.
\end{proof}

\section*{Acknowledgements}

I want to express my thanks to the referee for the {\it Annales math\'ematiques du Qu\'ebec}
for her/his valuable comments and remarks.
Also, many thanks to A.~Chapman and D.~Neftin for their interest and support,
to D.~Riley for the thoughtful discussions about the Zassenhaus filtration and restricted Lie algebras, 
and to S.K.~Chebolu and J.~Min\'a\v{c} for working with me on $p$-rigid fields.

\begin{center}
 \includegraphics[scale=0.095]{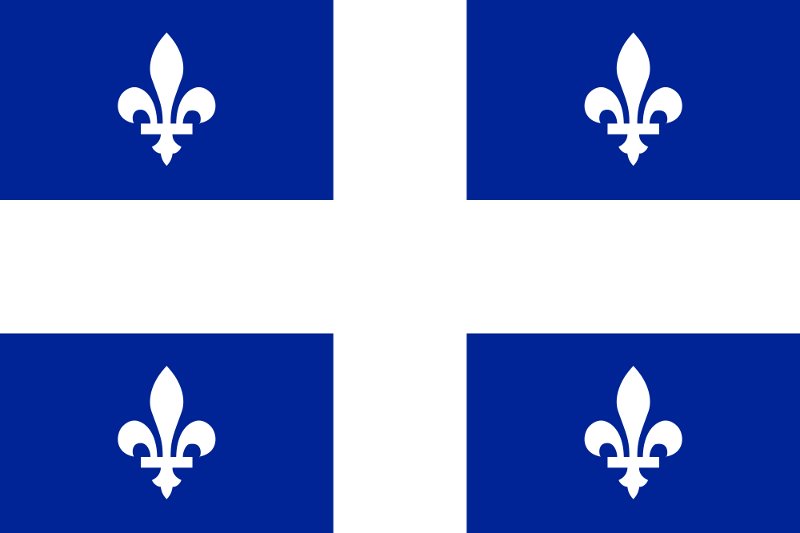}

\end{center}


\end{document}